 \documentclass[a4paper,11pt]{article}
 \pdfoutput=1
\usepackage{pdfsync}

 \usepackage[plainpages=false]{hyperref}
 \usepackage{amsfonts,amsmath,amssymb,amsthm}
 \usepackage{latexsym,lscape,rawfonts,mathrsfs}
 \usepackage{mathtools}
 \usepackage[utf8]{inputenc}
 \usepackage[overload]{empheq}
 \usepackage{cases}

 \usepackage[dvips]{color}
 \usepackage{multicol}

 \usepackage[all]{xy}
 \usepackage{makeidx}
 \usepackage{graphicx,psfrag}
 \usepackage{pstool}
 \usepackage{float}

 \usepackage{array,tabularx}

 \usepackage{setspace}

 \usepackage[titletoc]{appendix}

\usepackage{txfonts}

\newtheorem{prop}{Proposition}[section]

\newtheorem{proposition}[prop]{Proposition}

\newtheorem{theorem}[prop]{Theorem}

\newtheorem{corollary}[prop]{Corollary}

\newtheorem{remark}[prop]{Remark}

\numberwithin{equation}{section}

\title{Boundedness and gradient estimates for solutions to $\Delta u + a(x)u\log u + b(x)u = 0$ on Riemannian manifolds}
\author{Jie Wang and Youde Wang}
\author{Jie Wang \; and\; Youde Wang\footnote{corresponding author}}
\date{}

\begin{document}	
	\maketitle
	\begin{abstract}
		In this paper, combining Nash-Moser iteration and Sallof-Coste type Sobolev ineualities, we establish fundamental and concise $C^0$ and $C^1$ estimates for solutions to a class of nonlinear elliptic equations of the form
		$$\Delta u(x)+a(x)u(x)\ln u(x)+b(x)u(x)=0,$$
		which possesses abundant geometric backgrounds. Utilizing these estimates which retrieve more geometric information, we obtain some further properties of such solutions. Especially, we prove a local Liouville type theorem of corresponding constant coefficient equation.
	\end{abstract}
	\\
	
	\tableofcontents
	
	
	\section{Introduction}
	This is a further study based on our previous works \cite{Wj,WW}. In this paper, we continue to study the following equation on a smooth complete Riemannian n-manifold $(M^n,g)(n\geq 3)$:
	\begin{equation}\label{1.1}
		\Delta u(x)+a(x)u(x)\ln u(x)+b(x)u(x)=0,
	\end{equation}
	where $\Delta$ is the Laplace-Beltrami operator, $a(x)$ and $b(x)$ are coefficient functions on $(M^n,g)$. The existence of such solution $u$ on compact domains is a classical result proved by Rothaus \cite{Ro}.
	
	The type of equation (\ref{1.1}) is closely related to that of Euler-Lagrange equations of the $\mathcal{W}$ entropy and Log-Sobolev functional on $(M^n, g)$. $\mathcal{W}$ entropy \cite{P} is defined for $f\in W^{1,2}(M^n, e^{-f} dV)$ as
	$$\mathcal{W}(g, f, \tau) = \int_M[\tau(| \nabla f|^2 + R) + f - n](4\pi\tau)^{-\frac{n}{2}}e^{-f} dV,$$
	where $\tau > 0$ is a scale parameter and $R \in L^1(M^n, e^{-f} dV)$ is supposed. Setting $u^2 = (4\pi\tau)^{-\frac{n}{2}}e^{-f}$,
	we may rewrite the above $\mathcal{W}$-entropy for $u \in W^{1,2}(M^n,g)$ as
	$$\mathcal{W}(g, u, \tau)=\int_M\left(\tau(4|\nabla u|^2 + Ru^2) - u^2\log u^2 - nu^2 -\frac{n}{2}\log(4\pi\tau)u^2\right)dV.$$
	By the scale invariance property of $\mathcal{W}$ entropy, without loss of generality, we may assume $\tau=1$. Then, it follows
	$$\mathcal{W}(g, u, 1) = \mathcal{L}(u, g) - \frac{n}{2}(\log(4\pi)+2),$$
	if $\|u\|_{L^2(M^n,g)}=1$. Here $\mathcal{L}(u, g)=:\mathcal{L}(u, M^n, g)$ is the Log-Sobolev functional on $(M^n, g)$ perturbed by $R$, the scalar curvature of the manifold $(M^n, g)$, which is defined for $u\in W^{1,2}(M^n,g)$ as
	$$\mathcal{L}(u, M^n, g)=\int_M (4|\nabla u|^2 + Ru^2 - u^2\log u^2)dV.$$
	We define the best Log-Sobolev constant of a domain $\Omega\subset M^n$ as
	$$\lambda(\Omega) = \inf\{\mathcal{L}(u, g): u\in C_c^\infty(\Omega),\,\,\|u\|_{L^2(\Omega)}=1\}.$$
	When $\Omega = M^n$, we denote by $\lambda(M)$ the best Log-Sobolev constant of $(M^n, g)$. The Euler-Lagrange equation for the Log-Sobolev functional $\mathcal{L}$ is given by
	\begin{equation}\label{1.2*}
		\Delta u + \frac{1}{2}u \log u + \frac{1}{4}(\lambda- R)u =0.
	\end{equation}
	Under certain geometric conditions,  \cite{RV,Z*} proved some existence results of the extremal functions of Log-Sobolev functional defined on a complete non-compact Riemannian manifold. Especially, corresponding $C^0$ and $C^1$ estimates of $u$ are indispensable in their studies.
	
	Besides (\ref{1.2*}), after normalization, the potential function $f$ of a Ricci soliton also satisfies
	\begin{equation}\label{1.3*}
		\Delta\left( e^{-f}\right) -\varepsilon e^{-f}\log\left(e^{-f}\right) = 0,
	\end{equation}
	where $\varepsilon\in\left\lbrace-1,0,1\right\rbrace$ and the soliton is called shrinking, steady, expanding respectively (cf. \cite{CL}). As we know, Ricci solitons are crucial for the study of singularity analysis of Ricci flows. In short, (\ref{1.2*}) and (\ref{1.3*}) imply that equation (\ref{1.1}) possesses abundant geometric backgrounds.
	
	For constant $a(x)\equiv a$ and $b(x)\equiv0$, one has discussed the bounds of the solutions $u$ to equation (\ref{1.1}) via proving the Li-Yau type gradient estimates of corresponding parabolic equation and has obtained many interested results, for details we refer to \cite{C-C*, CL, Q1, Yang} and references therein. For equation (\ref{1.1}) with variable coefficients, J. Wang \cite{Wj} showed the bounds via elliptic gradient estimates and all of these estimates were based on the maximum principle.
	
	As for gradient estimates for the solution $u$ to equation (\ref{1.1}), most of the previous results relied on the maximum principle, such as \cite{CL, Q1, Q2, Yang}. Moreover, in most of the relevant work one chose $\ln u$ as a test function, but by taking a further observation, we find that $\ln u$ is not a good function used to derive the gradient estimates (cf. Proposition \ref{p3.3}). In the previous work \cite{WW} due to the authors of this paper, for a bounded positive solution $u$, under integral Ricci curvature conditions we employed the Nash-Moser iteration method to show some new local gradient estimates by analyzing the function $u^{\frac{1}{q}}(q>1)$ instead of $\ln u$. Especially, this type of gradient estimate was applied to study the properties of Ricci flow under integral Ricci curvature conditions by Ma-Wang \cite{MW} successfully.
	
	Recently, Y. Wang (one of the authors of this paper) and his coauthors considered $$\Delta_p u+ a u^q=0$$ in \cite{HWW, WWei} and combined Nash-Moser iteration method and Saloff-Coste's inequalities to derive some unified Cheng-Yau type inequalities (also see \cite{WZ}). Shortly after, inspired by \cite{WW}, Han, He and Wang in \cite{HHW} adopted a similar method with that in \cite{HWW, WWei} to appraoch the gradient estimates of solutions to some quasilinear elliptic equations, for instance
	$$\Delta_p u-|\nabla u|^q+b(x)|u|^{r-1}u=0$$
	which is defined on a complete Riemannian manifold $(M,g)$. In particular, in the case $b(x)\equiv0$, a unified Cheng-Yau type estimate of the solutions to this equation is derived.
	
	Naturally, one wants to know whether or not combining Saloff-Coste type Sobolev inequalities and Nash-Moser iteration leads to some more refined gradient estimates for solutions to \eqref{1.1}?
	
	The goal of the present paper is to answer the above question, i.e., we can obtain some general estimates of bounds on solutions to equation \eqref{1.1} by the method mentioned in the above. To compare with the maximum principle, we find that Nash-Moser iteration owns some unique advantages, for instance, one does not need to look for a test function painstakingly. Besides, this method still works for more general curvature conditions such as integral Ricci curvature conditions, for details we refer to \cite{WW} for related discussions.
	
	In order to state our results, we need to introduce some notations. In the following, we will use $B(x,r)$ or $B_r$ and $\left| B(x, r)\right|$ or $\left| B_r\right| $ to denote the geodesic ball with radius $r$ in $M$ centered at $x$ and its volume, respectively. $(f)^+$ denotes the positive part of $f$ and $(f)^-$ denotes the negative part of $f$. For convenience, for $p\geq1$, we define the average $L^p$ norm as
	$$\left| \left| f\right| \right|^*_{p,B(x,r)}= \left( \fint_{B(x,r)}\left| f\right|^p \right)^{\frac{1}{p}}.$$
	Furthermore, when $p\geq1$, it's well-known that the norm $\left| \left| f\right| \right|^*_{p,B(x,r)}$ is non-decreasing in $p$ for fixed $f$ and $B(x,r)$. We also assume that $\partial B_r$ does not intersect with the boundary $\partial M^n$.
	
	Now, we are ready to state our main results:
	
	\begin{theorem}\label{t1.1}
		Let $B_{2r}\subset M^n$ be some geodesic ball with $0<r\leq1$ and its Ricci curvature $Ric\geq-(n-1)Kg$ for some $K\geq0$. Let $u>0$ be a positive (weak) solution to equation (\ref{1.1}) on $B_{2r}$ and $p>\frac{n}{2}$.
		
		(1). Suppose $a(x)>0$, then there exists a constant
		$$C=C\left( n,K,p,\left| \left| a\right| \right|^*_{2p, B_r}, \left| \left| b^-\right| \right|^*_{p, B_r}, \left| \left| \frac{(\Delta a)^- }{a}\right| \right|^*_{p, B_r}, \left| \left|\nabla a\right| \right|^*_{2p, B_r},\left| \left|\nabla b\right| \right|^*_{2p, B_r}\right)>0$$
		such that on $B_{\frac{r}{2}}$,
		\begin{equation}\label{1.2}
			a\ln u\leq \frac{C}{r^2}.
		\end{equation}
		Especially, if $a\geq A_1>0$, then on $B_{\frac{r}{2}}$,
		\begin{equation}\label{1.3}
			u\leq e^{\frac{C}{A_1r^2}}.
		\end{equation}
		
		(2). (i) Suppose $a(x)<0$, then there exists a constant
		$$C=C\left( n,K,p,\left| \left| a\right| \right|^*_{2p, B_r}, \left| \left| b^-\right| \right|^*_{p, B_r}, \left| \left| \frac{(\Delta a)^+}{a}\right| \right|^*_{p, B_r}, \left| \left|\nabla a\right| \right|^*_{2p, B_r},\left| \left|\nabla b\right| \right|^*_{2p, B_r}\right)>0$$
		such that on $B_{\frac{r}{2}}$,
		\begin{equation}\label{1.4}
			a\ln u\leq \frac{C}{r^2}.
		\end{equation}
		Especially, if $a\leq A_2<0$, then on $B_{\frac{r}{2}}$,
		\begin{equation}\label{1.5}
			u\geq e^{\frac{C}{A_2r^2}}.
		\end{equation}
		
		(ii) Suppose $a<0$ is constant, then there exists a constant
		$$C=C\left( n,K,a,p,\left| \left|(\ln u)^+\right| \right|^*_{p, B_r}, \left| \left| b\right| \right|^*_{p, B_r}, \left| \left|\nabla b\right| \right|^*_{2p, B_r}\right)>0$$
		such that on $B_{\frac{r}{2}}$,
		\begin{equation*}
			\ln u\leq \frac{C}{r^2}.
		\end{equation*}
	\end{theorem}
	
	By the $C^0$ estimates obtained in the above Theorem \ref{t1.1} and the techniques used in our previous paper \cite{WW}, we are able to conclude the following general $C^1$ estimates which are helpful for further studies of the properties of solutions $u$.
	
	\begin{theorem}\label{t1.2}
		Let $B_{4r}\subset M^n$ be some geodesic ball with $0<r\leq1$ and its Ricci curvature $Ric\geq-(n-1)Kg$ for some $K\geq0$. Let $u>0$ be a positive solution to equation (\ref{1.1}) on $B_{2r}$, $p>\frac{n}{2}$ and $q>1$.
		
		(1). Suppose $a\geq A_1>0$, there exists a constant
		$$C=C\left( n,K,p,q,A_1, \left| \left| a\right| \right|^*_{2p, B_r}, \left| \left| b\right| \right|^*_{p, B_r}, \left| \left|(\Delta a)^- \right| \right|^*_{p, B_r}, \left| \left|\nabla a\right| \right|^*_{2p, B_r},\left| \left|\nabla b\right| \right|^*_{2p, B_r}\right)>0$$
		such that on $B_{\frac{r}{2}}$,
		\begin{equation}\label{1.6}
			\frac{\left|\nabla u\right|}{u^{1-\frac{1}{2q}}} \leq \frac{C}{r}.
		\end{equation}
		
		(2). Suppose $a\leq A_2<0$ and $u\leq D$, there exists a constant
		$$C=C\left( n,K,p,q,D,A_2, \left| \left| a\right| \right|^*_{2p, B_r}, \left| \left| b\right| \right|^*_{p, B_r}, \left| \left|(\Delta a)^+\right| \right|^*_{p, B_r}, \left| \left|\nabla a\right| \right|^*_{2p, B_r},\left| \left|\nabla b\right| \right|^*_{2p, B_r}\right)>0$$
		such that on $B_{\frac{r}{2}}$,
		\begin{equation}\label{1.7}
			\frac{\left|\nabla u\right|}{u^{1-\frac{1}{2q}}} \leq \frac{C}{r}.
		\end{equation}
	\end{theorem}
	
	If the coefficients $a(x)$ and $b(x)$ of equation (\ref{1.1}) are constants, one would expect better behaviors of $u$, such as Liouville type theorems. Moreover, in this situation, $u\equiv e^{-\frac{b}{a}}$ is the trivial positive solution.
	\begin{theorem}\label{t1.3}
		For constant coefficients $a>0$ and $b$, there exists a constant $\varepsilon=\varepsilon\left(n,a,b\right)>0$ such that if
		\begin{equation}\label{Seo}
			0<e^{-\frac{b}{a}}-\varepsilon\leq u\leq e^{-\frac{b}{a}}\quad\mbox{or}\quad e^{-\frac{b}{a}}\leq u\leq e^{-\frac{b}{a}}+\varepsilon
		\end{equation}
		is a solution to (\ref{1.1}) on $B\left(x,\varepsilon^{-1}\right)$ with $Ric\geq 0$, then $u\equiv e^{-\frac{b}{a}}$. Especially, in these situations, except for $u\equiv e^{-\frac{b}{a}}$, there is no nonconstant solution $u$, which satisfies the above pinching condition \eqref{Seo}, such that $u\rightarrow e^{-\frac{b}{a}}$ at infinity on a complete non-compact Riemannian manifold with $Ric\geq 0$.
	\end{theorem}
	
	This paper is organized as follows. We give the $C^0$ estimates and gradient estimates of solutions to  (\ref{1.1}) respectively in Section \ref{sec2} and Section \ref{sec3}. In section \ref{sec4}, we are devoted to study the global properties of the solutions, and prove mainly the above local Liouville type Theorem \ref{t1.3}.
	
	\section{$C^0$ estimates}\label{sec2}
	First we recall the Saloff-Coste's Sobolev inequality which is crucial for the  Nash-Moser iteration.
	\begin{theorem}[\cite{SC}, Theorem 3.1]\label{t2.1}
		Let $(M^n,g)$ be a complete Riemannian manifold with $Ric\geq-(n-1)K g$ on some $B_{2r}$. Then for $n\geq3$, there exists a dimensional constant $C_n$ such that for all $\varphi\in C^\infty_0(B_r)$,
		\begin{equation*}
			\left( \int_{B_r}\varphi^{\frac{2n}{n-2}}\right)^{\frac{n-2}{n}}\leq \frac{C_Sr^2}{\left|B_r\right|^{\frac{2}{n}}}\int_{B_r}\left(\left|\nabla \varphi \right|^2+r^{-2}\varphi^2\right),
		\end{equation*}
		where the Sobolev constant $$C_S=e^{C_n\left( 1+r\sqrt{K}\right)}.$$
		Moreover, for $n=2$, the above inequality still holds with $n$ replaced by any $n'>2$. For convenience, we also write the Sobolev inequality as
		\begin{equation}\label{2.1}
			\left( \fint_{B_r}\varphi^{\frac{2n}{n-2}}\right)^{\frac{n-2}{n}}\leq C_Sr^2\fint_{B_r}\left(\left|\nabla \varphi \right|^2+r^{-2}\varphi^2\right).
		\end{equation}
	\end{theorem}
	Here we remark that this type of Sobolev inequalities can be also derived by heat kernel estimates (cf. \cite[section 14]{Li}).
	
	Now we are going to prove Theorem \ref{t1.1}.
	\begin{proof}[\textbf{Proof of Theorem \ref{t1.1}}]
		\textbf{(1) }Let $w=\ln u$ and $G=\left|\nabla w\right|^2+aw+1$. By (\ref{1.1}), direct computations imply
		\begin{equation}\label{2.2}
			\Delta w+\left|\nabla w\right|^2+aw+b=0,
		\end{equation}
		and hence
		\begin{equation}\label{2.3}
			\Delta w+G+b-1=0.
		\end{equation}
		Without loss of generality, we may assume $w>0\, (i.e. \hspace*{0.5em}a(x)w(x)>0,\hspace*{0.5em}u(x)>1 \hspace*{0.5em}\mbox{and} \hspace*{0.5em}G(x)>1)$, since otherwise we can replace $w$ by $w^+$ in $G$. In the following, we only show the case of $n\geq 3$, and the case of $n=2$ can be proved by the same arguments.
		To apply Nash-Moser iteration, we need to estimate $\Delta G$. By (\ref{2.2}) and (\ref{2.3}),
		\begin{align}
			\Delta G&=\Delta\left|\nabla w\right|^2+\Delta(aw)\nonumber\\
			&=2\left|D^2w\right|^2+2\left\langle\nabla w,\nabla \Delta w\right\rangle+2Ric(\nabla w,\nabla w)+a\Delta w+w\Delta a+2\left\langle \nabla a,\nabla w\right\rangle\nonumber\\
			&\geq\frac{2(\Delta w)^2}{n}-2\left\langle\nabla w,\nabla (G+b)\right\rangle-2(n-1)K\left|\nabla w\right|^2-a(G+b-1)+w\Delta a-\left|\nabla a\right|^2-\left|\nabla w\right|^2\nonumber\\
			&\geq \frac{2(G+b-1)^2}{n}-2\left\langle\nabla w,\nabla G\right\rangle-\left|\nabla b\right|^2-(G-aw-1)-2(n-1)K(G-aw-1)\nonumber\\
			&\quad-a(G+b-1)+w\Delta a-\left|\nabla a\right|^2-(G-aw-1).\nonumber
		\end{align}
		After recombining,
		\begin{align}
			\Delta G&\geq\frac{2G^2+2(b-1)^2}{n}-\left\lbrace 2(n-1)K+a+2-\frac{4(b-1)}{n}\right\rbrace G- 2\left\langle\nabla w,\nabla G\right\rangle\nonumber\\
			&\quad +\left\lbrace 2(n-1)Ka+2a+\Delta a\right\rbrace w-\left|\nabla a\right|^2-\left|\nabla b\right|^2-ab+a\nonumber\\
			&\geq\frac{2G^2}{n}-\left\lbrace 2(n-1)K+a+2-\frac{4(b-1)}{n}\right\rbrace G- 2\left\langle\nabla w,\nabla G\right\rangle\nonumber\\
			&\quad +\left\lbrace 2(n-1)Ka+2a+\Delta a\right\rbrace w-\left|\nabla a\right|^2-\left|\nabla b\right|^2+\left(\sqrt{\frac{2}{n}}(b-1)-\frac{1}{2}\sqrt{\frac{n}{2}}a\right)^2-\frac{n}{8}a^2.\label{2.4}
		\end{align}
		Since $a>0,0<aw<G, G>1$ and $K\geq0$, then by (\ref{2.4}),
		\begin{align}
			\Delta G&\geq\frac{2G^2}{n}-\left\lbrace 2(n-1)K+a-\frac{4b-4}{n}+2+\frac{(\Delta a)^- }{a}\right\rbrace G\nonumber\\
			&- 2\left\langle\nabla w,\nabla G\right\rangle-\left|\nabla a\right|^2-\left|\nabla b\right|^2-\frac{n}{8}a^2.\label{2.5}
		\end{align}
		Next, for any $l\geq0$ and $\eta\in C^\infty_0(B_{r})$, we multiply by $\eta^2G^l$ on both sides of (\ref{2.5}) and integrate on $B_{r}$, then
		\begin{align}
			&\int\eta^2G^l\Delta G\nonumber\\
			&\geq\int\eta^2G^l\left( \frac{2G^2}{n}-\left\lbrace 2(n-1)K+a-\frac{4b-4}{n}+2+\frac{(\Delta a)^- }{a}\right\rbrace G- 2\left\langle\nabla w,\nabla G\right\rangle\right) \nonumber\\
			&\quad-\int\eta^2G^{l+1}\left( \left|\nabla a\right|^2+\left|\nabla b\right|^2+\frac{n}{8}a^2\right).\label{2.6}
		\end{align}
		By Green formula, we have
		$$\int\eta^2G^{l+1}\Delta w=-\int \left\langle \nabla\left( \eta^2G^{l+1}\right), \nabla w \right\rangle =-\int (l+1)\eta^2G^l\left\langle \nabla G, \nabla w\right\rangle-\int 2\eta G^{l+1}\left\langle \nabla\eta, \nabla w\right\rangle,$$
		therefore
		\begin{align}
			&\int\eta^2G^l\left\langle \nabla G, \nabla w\right\rangle\nonumber\\&=-\frac{1}{l+1}\int\eta^2G^{l+1}\Delta w-\frac{2}{l+1}\int G^{l+1}\left\langle \nabla\eta, \eta\nabla w\right\rangle\nonumber\\
			&\leq -\frac{1}{l+1}\int\eta^2G^{l+1}\left(-G-b+1\right)+\frac{1}{l+1}\int G^{l+1}\left(\left|\nabla\eta \right|^2 +\eta^2\left|\nabla w \right|^2\right)\nonumber\\
			&\leq \frac{1}{l+1}\int\eta^2G^{l+2}+\frac{1}{l+1}\int \left( b\eta^2+\left|\nabla\eta \right|^2\right) G^{l+1}+\frac{1}{l+1}\int\eta^2G^{l+1}(G-aw-1)\nonumber\\
			&\leq \frac{2}{l+1}\int\eta^2G^{l+2}+\frac{1}{l+1}\int b\eta^2G^{l+1}+\frac{1}{l+1}\int G^{l+1}\left|\nabla\eta \right|^2.\label{2.7}
		\end{align}
		Combining (\ref{2.6}) and (\ref{2.7}) yields
		\begin{align}\label{2.8}
			&\int\eta^2G^l\Delta G\nonumber\\
			&\geq \int \left(\frac{2}{n}-\frac{4}{l+1}\right)\eta^2G^{l+2}-\left(2(n-1)K+2+a+\frac{(\Delta a)^- }{a}+\frac{2b}{l+1} -\frac{4b-4}{n}\right)\eta^2G^{l+1}\nonumber\\
			&\quad-\int\left( \left|\nabla a\right|^2+\left|\nabla b\right|^2+\frac{n}{8}a^2\right)\eta^2G^{l+1}-\frac{2}{l+1}\int\left|\nabla\eta \right|^2G^{l+1}.
		\end{align}
		Next we let $l\geq 2n-1$, then by (\ref{2.8}),
		\begin{align}\label{2.9}
			&\int\eta^2G^l\Delta G\nonumber\\
			&\geq-\left(2(n-1)K+2+a+\frac{4b^-+4}{n}+\frac{(\Delta a)^- }{a}+\left|\nabla a\right|^2+\left|\nabla b\right|^2+\frac{n}{8}a^2\right)\eta^2G^{l+1}\nonumber\\
			&\quad-\frac{2}{l+1}\int\left|\nabla\eta \right|^2G^{l+1}.
		\end{align}
		By (5.11) of \cite{DWZ}, utilizing integration by parts, we have the following general integral inequality,
		\begin{align}\label{2.10}
			&\int\left| \nabla\left( \eta G^{\frac{l+1}{2}}\right) \right|^2\nonumber\\
			&\leq-\frac{\left( l+1\right)^2 }{2l}\int\eta^2G^l\Delta G+\frac{\left( l+1\right)^2+l }{l^2} \int G^{l+1}\left| \nabla\eta\right|^2-\frac{l+1}{l}\int \eta G^{l+1}\Delta \eta.
		\end{align}
		Notice that
		$$\int\left| \nabla\left( \eta G^{\frac{l+1}{2}}\right) \right|^2=\int\left(\eta^2\left| \nabla\left(G^{\frac{l+1}{2}}\right) \right|^2+\left|\nabla\eta\right|^2G^{l+1}+2\eta G^{\frac{l+1}{2}}\left\langle\nabla\eta,\nabla\left(G^{\frac{l+1}{2}}\right)\right\rangle\right),$$
		\begin{align*}
			-\int \eta G^{l+1}\Delta \eta=\int\left\langle \nabla\eta,\nabla\left(\eta G^{l+1} \right)\right\rangle&=\int\left(  \left| \nabla\eta\right|^2G^{l+1}+\eta\left\langle\nabla\eta,\nabla G^{l+1} \right\rangle \right)\\
			&=\int\left(  \left| \nabla\eta\right|^2G^{l+1}+2\eta G^{\frac{l+1}{2}}\left\langle\nabla\eta,\nabla\left(G^{\frac{l+1}{2}}\right)\right\rangle\right).
		\end{align*}
		Also,
		\begin{align*}
			&\int\left| \nabla\left( \eta G^{\frac{l+1}{2}}\right) \right|^2+\frac{l+1}{l}\int2\eta G^{\frac{l+1}{2}}\left\langle\nabla\eta,\nabla\left(G^{\frac{l+1}{2}}\right)\right\rangle\\
			&=\int\left(\eta^2\left| \nabla\left(G^{\frac{l+1}{2}}\right) \right|^2+\left|\nabla\eta\right|^2G^{l+1}-\frac{2}{l}\eta G^{\frac{l+1}{2}}\left\langle\nabla\eta,\nabla\left(G^{\frac{l+1}{2}}\right)\right\rangle\right)\\
			&\geq\frac{l-1}{l}\int\left| \nabla\left( \eta G^{\frac{l+1}{2}}\right) \right|^2.
		\end{align*}
		Combining these relational expressions and (\ref{2.10}), we have
		\begin{align}\label{2.11}
			&\int\left| \nabla\left( \eta G^{\frac{l+1}{2}}\right) \right|^2\nonumber\\
			&\leq-\frac{\left( l+1\right)^2 }{2(l-1)}\int\eta^2G^l\Delta G+\frac{\left( l+1\right)^2+l(l+2)}{l(l-1)} \int G^{l+1}\left| \nabla\eta\right|^2.
		\end{align}
		Then by (\ref{2.11}) and the assumption $G>1$, we infer from (\ref{2.9}) that
		\begin{align}
			&\int\left| \nabla\left( \eta G^{\frac{l+1}{2}}\right) \right|^2\nonumber\\
			&\leq\frac{\left( l+1\right)^2 }{2(l-1)}\int \left(2(n-1)K+2+a+\frac{4b^-+4}{n}+\frac{(\Delta a)^- }{a}+\left|\nabla a\right|^2+\left|\nabla b\right|^2+\frac{n}{8}a^2\right)\eta^2G^{l+1}\nonumber\\
			&\quad+\frac{l\left(2l+3\right)+(l+1)^2}{l(l-1)}\int \left| \nabla\eta\right|^2G^{l+1}.\label{2.12}
		\end{align}
		By the Sobolev inequality (\ref{2.1}) with $\varphi=\eta G^{\frac{l+1}{2}}$ and $0<r\leq1$, (\ref{2.11}) gives
		\begin{align}\label{2.13}
			&\left( \fint_{B_r}\left(\eta^2 G^{l+1}\right)^{\frac{n}{n-2}}\right)^{\frac{n-2}{n}}\nonumber\\
			&\leq C_Sl\fint_{B_r}\left(2(n-1)K+3+a+\frac{4b^-+4}{n}+\frac{(\Delta a)^- }{a}+\left|\nabla a\right|^2+\left|\nabla b\right|^2+\frac{n}{8}a^2\right) \eta^2G^{l+1}\nonumber\\
			&\quad+C_Sr^2l\fint_{B_r}\left| \nabla\eta\right|^2G^{l+1}.
		\end{align}
		Furthermore, by interpolation inequality(cf. \cite[section 3]{WW}) with $p>\frac{n}{2}$ and $\mu=\frac{n}{n-2}$, there exists a constant $C_2=C_2\left(n,p\right)$ such that
		\begin{equation}\label{2.13*}
			C_Sl\fint a\eta^2G^{l+1}\leq\frac{1}{10}\left(\fint\left( \eta^2G^{l+1}\right)^\mu\right)^{\frac{1}{\mu}}+C_2\left(C_Sl \left| \left| a\right| \right|^*_{p, B_r}\right)^{\frac{2p}{2p-n}}\fint\eta^2G^{l+1}.
		\end{equation}
		Applying similar interpolation inequalities to the remainder terms and plugging these inequalities into (\ref{2.13}) implies that there exists a constant $$C_3=C_3\left( C_2,n,p,K,\left| \left| a\right| \right|^*_{2p, B_r}, \left| \left| b^-\right| \right|^*_{p, B_r}, \left| \left| \frac{(\Delta a)^- }{a}\right| \right|^*_{p, B_r}, \left| \left|\nabla a\right| \right|^*_{2p, B_r},\left| \left|\nabla b\right| \right|^*_{2p, B_r}\right)$$
		such that
		\begin{align}
			\left( \fint_{B_r}\left(\eta^2 G^{l+1}\right)^{\frac{n}{n-2}}\right)^{\frac{n-2}{n}}\leq C_3\left(C_Sl\right)^{\frac{2p}{2p-n}}\fint\eta^2v^{l+1}+10C_Sr^2l\fint_{B_r}\left| \nabla\eta\right|^2G^{l+1}.\label{2.14}
		\end{align}
		Here we used the monotonic inequality
		$$\left| \left|f\right| \right|^*_{p, B_r}\leq\left| \left| f\right| \right|^*_{2p, B_r}.$$
		
		Let $\phi(s)$ be a non-negative $C^2$-smooth function on $\left[0,+\infty \right)$ such that $\phi(s)=1$ for $s\leq \frac{1}{2}$, $\phi(s)=0$ for $ s\geq1$, and $-4\leq\phi'\leq0$. Next, let $\eta(y)=\phi\left(\frac{d(y,x)}{r}\right)$ where $d(y,x)$ denotes the distance from $y$ to $x$ and it is obvious that $\eta(y)$ is supported in $B_r$:
		\begin{align*}
			&\eta|_{B_\frac{r}{2}}=1,\\
			&\eta|_{M\backslash B_r}=0.
		\end{align*}
		Clearly,
		\begin{align*}
			\left| \nabla  \eta\right|^2\leq\frac{16}{r^2}.
		\end{align*}
		Recall $l\geq 2n-1$, then by standard iteration(e.g. see \cite[section 3]{WW}), for some constant $C_4=C_4(n,C_3,C_S,p)$,
		\begin{equation}\label{2.15}
			\sup\limits_{B_\frac{r}{2}}G\leq C_5\left\|G\right\|^*_{2n-1, B_r}.
		\end{equation}
		Since $Ric\geq-(n-1)Kg$, by Bishop-Gromov volume comparison theorem, the geodesic ball $B_r$ satisfies the volume doubling property, hence by the standard trick of lowering power(cf. \cite[Lemma 2.5]{WW}), (\ref{2.15}) can be improved to
		\begin{equation}\label{2.16}
			\sup\limits_{B_\frac{r}{2}}G\leq C_5\left\|G\right\|^*_{1, B_\frac{4r}{5}},
		\end{equation}
		where $C_5=C_5\left(C_3,C_4,C_S,n,K\right).$
		
		Next, we choose some $\eta\in C^\infty_0(B_r)$ such that
		$$\eta\equiv1\quad\mbox{on}\hspace*{0.3em}B_\frac{4r}{5}\quad\mbox{and}\quad \left| \nabla\eta\right|\leq\frac{10}{r}.$$
		For $\eta^2G$, since $aw>0$ and by (\ref{2.3}), there holds
		\begin{align*}
			\int\eta^2G=-\int\eta^2\left( \Delta w+b-1\right)&=2\int\left\langle \nabla\eta, \eta\nabla w\right\rangle-\int\eta^2(b-1)\nonumber\\
			&\leq2\int \left| \nabla\eta\right|^2+\frac{1}{2}\int\eta^2\left| \nabla w\right|^2-\int\eta^2(b-1)\nonumber\\
			&=2\int\left|\nabla\eta\right|^2+\frac{1}{2}\int\eta^2G-\frac{1}{2}\int\eta^2(aw+1)-\int\eta^2(b-1)\\
			&\leq 2\int \left| \nabla\eta\right|^2+\frac{1}{2}\int\eta^2G+\int\eta^2b^-+\frac{1}{2}\int\eta^2.
		\end{align*}
		Therefore,
		\begin{align}\label{2.17}
			\int\eta^2G\leq4\int \left| \nabla\eta\right|^2+2\int\eta^2b^-+\int\eta^2.
		\end{align}
		As a consequence of (\ref{2.17}),
		\begin{align}\label{2.18}
			0<\left\|G\right\|^*_{1, B_\frac{4r}{5}}\leq\frac{\int\eta^2G}{\left| B_\frac{4r}{5}\right|}\leq\frac{\left( 400+2\left| \left| b^-\right| \right|^*_{p, B_r}+1\right)\left| B_r\right|}{\left| B_\frac{4r}{5}\right|r^2}.
		\end{align}
		Finally,  by the definition of $G$, combining the volume comparison theorem and (\ref{2.18}) completes the proof.
		
		\textbf{(2)} (i) The proof is almost the same as above with a few minor modifications. Let $w=\ln u$ and $G=\left|\nabla w\right|^2+aw+1$. Without loss of generality, we may assume $w<0(i.e. \hspace*{0.5em}a(x)w(x)>0,\hspace*{0.5em}u(x)<1 \hspace*{0.5em}\mbox{and} \hspace*{0.5em}G(x)>1)$. Then the estimate (\ref{1.4}) follows verbatim.
		
		As for (ii), we need a few minor modifications based on the proof as above. Here for simplicity, we only consider the case of $a(x)<0$ is constant. Let $w=\ln u$ and $G=\left|\nabla w\right|^2-aw+1$. By (\ref{1.1}), direct computations imply
		\begin{equation*}
			\Delta w+\left|\nabla w\right|^2+aw+b=0,
		\end{equation*}
		and hence
		\begin{equation*}
			\Delta w+G+2aw+b-1=0.
		\end{equation*}
		Without loss of generality, we may assume $w>0(i.e. \hspace*{0.5em}a(x)w(x)<0,\hspace*{0.5em}u(x)>1 \hspace*{0.5em}\mbox{and} \hspace*{0.5em}G(x)>1)$. Then the same arguments as above give the desired estimate of the upper bound of $\ln u$. The only difference is that during the iterative process, we need $\left\|w^+\right\|^*_{p, B_r}$ to control the term $-\fint_{B_r}w\eta^2G^{l+1}$ as we did in (\ref{2.13*}).
	\end{proof}

	\begin{remark}
		Since all the elements we need, such as the Sobolev inequality and comparison geometry, have counterparts under integral Ricci curvature conditions, hence the above local $C^0$ and the following $C^1$ estimates of the solution $u$ are still valid under integral Ricci curvature conditions. See \cite{WW} for related discussions.
	\end{remark}

	\section{$C^1$ estimates}\label{sec3}
	\begin{proof}[\textbf{Proof of Theorem \ref{t1.2}}]
		\textbf{(1)}
		By Theorem \ref{t1.1}, there exists a constant $D>0$ such that $u\leq D$ on $B_r$. Set $w=u^{\frac{1}{q}}$ for some $q>1$, then by (\ref{1.1}), $w$ satisfies
		\begin{equation}\label{3.1}
			\Delta w+(q-1)\frac{\left| \nabla w\right|^2}{w}+aw\log w+\frac{bw}{q}=0.
		\end{equation}
		Let $v=(q-1)\frac{\left| \nabla w\right|^2}{w}$ and $b'=\frac{b}{q}$. Without loss of generality, we may assume $v\geq1$, since otherwise we may consider $v+1$ instead of $v$.
		By Bochner formula,
		\begin{align}
			\Delta v=&(q-1)\Delta\left( \frac{\left| \nabla w\right|^2}{w}\right)\nonumber\\
			=&(q-1)\left[ \frac{\Delta\left| \nabla w\right|^2}{w}+\left| \nabla w\right|^2\Delta\left( \frac{1}{w}\right)+2\left\langle \nabla(\left| \nabla w\right|^2), \nabla\left(\frac{1}{w} \right) \right\rangle  \right] \nonumber\\
			\geq&(q-1)\left[ \frac{2\left| D^2 w\right|^2+2\left\langle\nabla w, \nabla\Delta w \right\rangle - 2(n-1)K\left| \nabla w\right|^2 }{w}+\left| \nabla w\right|^2\left( -\frac{\Delta w}{w^2}+\frac{2\left| \nabla w\right|^2}{w^3}\right)\right]\nonumber\\
			& +2(q-1)\left\langle \nabla\left| \nabla w\right|^2, -\frac{\nabla w}{w^2} \right\rangle.\label{3.2}
		\end{align}
		Substituting the equations $\Delta w=-v-aw\log w-b'w$ and $\left| \nabla w\right|^2=\frac{vw}{q-1}$ into (\ref{3.2}) gives
		\begin{align}\label{3.3}
			\Delta v\geq&(q-1)\left[ \frac{2\left| D^2 w\right|^2}{w}+\frac{v^2}{(q-1)w}-\left( \frac{2}{w}+\frac{2}{(q-1)w}\right)\left\langle\nabla w,\nabla v \right\rangle\right]\nonumber\\
			&-2(q-1)\left( \ln w\left\langle\nabla a,\nabla w\right\rangle+\left\langle \nabla w, \nabla b'\right\rangle\right)
			-\left(2(n-1)K+a\log w+b'+2a\right) v.
		\end{align}
		Moreover,
		$$-2\left\langle \nabla w, \nabla b'\right\rangle\geq-\frac{\left| \nabla w\right|^2}{w}-w\left| \nabla b'\right|^2,$$
		\begin{align*}
			-2\ln w\left\langle\nabla a,\nabla w\right\rangle\geq-2 \left|\ln w\right|\left|\nabla a\right|\left|\nabla w\right|&=-2\left|\nabla a\right|\left|\ln w\right|\sqrt{\frac{wv}{q-1}}\\
			&=-2\left|\nabla a\right|\left(\sqrt{w}\left|\ln w\right|\right)\sqrt{\frac{v}{q-1}}\\
			&\geq-\frac{\max\left\lbrace 2e^{-1},D^{\frac{1}{2}}\left| \ln D\right| \right\rbrace }{(q-1)^{\frac{1}{2}}}\left(\left|\nabla a\right|^2+v\right)\\
			&=:-C_1\left(\left|\nabla a\right|^2+v\right),
		\end{align*}
		
		hence (\ref{3.3}) with $v\geq1$ leads to
		\begin{align}\label{3.4}
			\Delta v\geq&(q-1)\left[ \frac{v^2}{(q-1)w}-\left( \frac{2}{w}+\frac{2}{(q-1)w}\right)\left\langle\nabla w,\nabla v \right\rangle\right]-\left\lbrace C_2+a\left(2+\ln D\right)+b'\right\rbrace v\nonumber\\
			&-2\left(q-1\right)\left(D^{\frac{1}{q}}\left| \nabla b'\right|^2+C_1\left| \nabla a\right|^2\right)v,
		\end{align}
		where $C_2=C_2\left(n,q,K,D,C_1\right)$.
		
		Next, for any $l\geq0$ and $\eta\in C^\infty_0\left( B_r\right) $, we multiply by $\eta^2v^l $ on the both sides of (\ref{3.4}), then
		\begin{align}\label{3.5}
			\int\eta^2v^l\Delta v\geq&\int\left( \frac{\eta^2v^{l+2}}{w}-2q\eta^2v^l\left\langle\nabla \log w, \nabla v \right\rangle-\left\lbrace C_2+a\left(2+(\ln D)^+\right)+(b')^+\right\rbrace \eta^2v^{l+1}\right)\nonumber\\
			&-2\int\left(q-1\right)\left(D\left| \nabla b'\right|^2+C_1\left| \nabla a\right|^2\right)\eta^2v^{l+1}.
		\end{align}
		By Green‘s formula and Cauchy-Schwartz inequality,
		\begin{align}
			&\int\eta^2v^l\left\langle \nabla v, \nabla \log w\right\rangle\nonumber\\
			&=-\frac{1}{l+1}\int\eta^2v^{l+1}\Delta \log w-\frac{2}{l+1}\int v^{l+1}\left\langle \nabla\eta, \eta\nabla \log w\right\rangle\nonumber\\
			&\leq -\frac{1}{l+1}\int\eta^2v^{l+1}\left( \frac{\Delta w}{w}-\frac{\left| \nabla w\right|^2}{w^2}\right)+\frac{1}{l+1}\int v^{l+1}\left(\left|\nabla\eta \right|^2 +\eta^2\frac{\left| \nabla w\right|^2}{w^2}\right)\label{3.6}.
		\end{align}
		Substituting the equations $\Delta w=-v-aw\log w-b'w$ and $\left| \nabla w\right|^2=\frac{vw}{q-1}$ into (\ref{3.6}) gives
		\begin{equation}\label{3.7}
			\begin{split}
				&\int\eta^2v^l\left\langle \nabla v, \nabla \log w\right\rangle\\
				&\leq\left( \frac{1}{l+1}+\frac{2}{(l+1)(q-1)}\right)\int\frac{\eta^2v^{l+2}}{w}+\frac{1}{l+1}\left( \int C_3\eta^2v^{l+1}+\int\left| \nabla\eta \right|^2v^{l+1}\right),
			\end{split}
		\end{equation}
		where $$C_3=\left( a\log D^{\frac{1}{q}}\right)^++(b')^+.$$
		Combining (\ref{3.5}) and (\ref{3.7}), for some constant $C_4=C_4(q,D,C_1,C_2)$,
		\begin{align*}
			&\int\eta^2v^l\Delta v\nonumber\\
			&\geq\int\left[1-\left(\frac{2q}{l+1}+\frac{4q}{(l+1)(q-1)} \right)\right]\frac{\eta^2v^{l+2}}{w}\\
			&\quad-C_4\left( \int \left(a+(b)^++\left| \nabla a\right|^2+\left| \nabla b\right|^2\right)\eta^2v^{l+1}+\int\left| \nabla\eta \right|^2v^{l+1}\right).
		\end{align*}
		Let $l+1\geq2q+\frac{4q}{q-1}$, then $1-\left(\frac{2q}{l+1}+\frac{4q}{(l+1)(q-1)}\right)\geq0$, and hence
		\begin{equation}\label{3.8}
			\int\eta^2v^l\Delta v\geq-C_4\left( \int \left(a+(b)^++\left| \nabla a\right|^2+\left| \nabla b\right|^2\right)\eta^2v^{l+1}+\int\left| \nabla\eta \right|^2v^{l+1}\right).
		\end{equation}
		Plugging (\ref{3.8}) into the integral inequality (\ref{2.11}) implies
		\begin{align*}
			&\int\left| \nabla\left( \eta v^{\frac{l+1}{2}}\right) \right|^2\nonumber\\
			&\leq\frac{\left(l+1\right)^2C_4}{2(l-1)}\int  \left(a+(b)^++\left| \nabla a\right|^2+\left| \nabla b\right|^2\right) \eta^2v^{l+1}+\frac{(lC_4+2)(l+1)^2+2l(l+2)}{2l(l-1)}\int\left| \nabla\eta\right|^2v^{l+1}.
		\end{align*}
		Then following the same iteration steps as to get (\ref{2.16}), there exists a constant $$C_5=C_5\left(n,p,q,K,C_S,C_4,\left| \left| a\right| \right|^*_{p, B_r},\left| \left| b^+\right| \right|^*_{p, B_r}, \left| \left|\nabla a\right| \right|^*_{2p, B_r}, \left| \left|\nabla b\right| \right|^*_{2p, B_r}\right)>0$$
		such that
		\begin{equation}\label{3.10}
			\left| \left| v\right| \right|_{\infty,B_{\frac{r}{2}}}\leq C_5\left| \left| v\right| \right|^*_{1,B_\frac{4r}{5}}.
		\end{equation}
		
		Next, we choose some $\eta\in C^\infty_0(B_r)$ such that
		$$\eta\equiv1\quad\mbox{on}\hspace*{0.3em}B_\frac{4r}{5}\quad\mbox{and}\quad \left| \nabla\eta\right|\leq\frac{10}{r}.$$
		For $\eta^2v$, there holds
		\begin{equation}\label{3.11}
			\begin{split}
				\int\eta^2v=-\int\eta^2\left( \Delta w+aw\log w+b'w\right).
			\end{split}
		\end{equation}
		By Green's formula,
		$$-\int\eta^2\Delta w=2\int\left\langle \nabla\eta, \eta\nabla w\right\rangle\leq\frac{1}{2}\int\eta^2v+2\int\frac{w\left| \nabla\eta\right|^2}{q-1}\leq\frac{1}{2}\int\eta^2v+\frac{200D^{\frac{1}{q}}}{(q-1)r^2}\left| B_r\right|.$$
		On the other hand, we also have
		$$-\int\eta^2(aw\log w+b'w)\leq\left( \frac{\left| \left| a\right| \right|^*_{p, B_r}}{e}+\frac{\left| \left| b^-\right| \right|^*_{p, B_r}D^{\frac{1}{q}}}{q}\right)\left|B_r \right|.$$
		Here, we have used the following facts that for $w\in(0, +\infty)$ there holds
		$$-w\log w\leq\frac{1}{e}.$$
		Now, by (\ref{3.11}) and $0<r\leq1$, we estimate
		\begin{align}\label{3.12}
			\left| \left| v\right| \right|^*_{1,B_\frac{4r}{5}}\leq\frac{\int\eta^2v}{\left|B_{\frac{4r}{5}}\right|}&\leq\frac{\left(\frac{400D^{\frac{1}{q}}}{(q-1)r^2}+\frac{2\left| \left| a\right| \right|^*_{p, B_r}}{e}+\frac{2\left| \left|b^-\right| \right|^*_{p, B_r}D^{\frac{1}{q}}}{q}\right)\left|B_r \right|}{\left|B_{\frac{4r}{5}} \right|}\nonumber\\
			&\leq\frac{\left(\frac{400D^{\frac{1}{q}}}{(q-1)}+\frac{2\left| \left| a\right| \right|^*_{p, B_r}}{e}+\frac{2\left| \left| b^-\right| \right|^*_{p, B_r}D^{\frac{1}{q}}}{q}\right)\left|B_r \right|}{r^2\left|B_{\frac{4r}{5}} \right|}.
		\end{align}
		Substituting (\ref{3.12}) into (\ref{3.10}) and noticing $$\frac{\left|\nabla w \right|^2}{w}=\frac{\left|\nabla u \right|^2u^{\frac{1}{q}-2}}{q^2},$$
		then by volume comparison theorem, we get the desired estimate
		\begin{equation*}
			\sup\limits_{B\left( x,\frac{r}{2}\right) }\frac{\left|\nabla u\right|}{u^{1-\frac{1}{2q}}} \leq \frac{C}{r},
		\end{equation*}
		where $C=C\left( n,K,p,q,A_1, \left| \left| a\right| \right|^*_{2p, B_r}, \left| \left| b\right| \right|^*_{p, B_r}, \left| \left|(\Delta a)^- \right| \right|^*_{p, B_r}, \left| \left|\nabla a\right| \right|^*_{2p, B_r},\left| \left|\nabla b\right| \right|^*_{2p, B_r}\right)$.
		
		\textbf{(2)} The proof is the same as above with a few minor modifications and the extra condition $u\leq D$ is used to control $\left\langle\nabla b,\nabla w\right\rangle$ and $\ln w\left\langle\nabla a,\nabla w\right\rangle$.
	\end{proof}
	
	Due to the previous $C^0$, $C^1$ estimates and standard bootstrapping arguments for elliptic PDEs, we immediately have following regularity of $u$.
	\begin{corollary}\label{c3.1}
		Let $u>0$ be a positive (weak) solution to the equation (\ref{1.1}) on  $B_r\subset(M^n,g)$ with $a(x), b(x)\in C^\infty\left(B_r\right)$. Suppose $a(x)>0$, or $a(x)<0$ and $u$ is bounded from above, then $u\in C^\infty\left(B_{r'}\right)$ for any $B_{r'}\subset\subset B_r$.
	\end{corollary}
	\begin{proof}
		By our estimates of Theorem \ref{t1.1} and \ref{t1.2}, we know $u, \left|\nabla u\right|,u\ln u,\left| \nabla(u\ln u)\right|$ are all locally bounded, where for $q>1$, we estimate
		$$\left| \nabla(u\ln u)\right|\leq\left|\ln u\right| \left| \nabla u\right|+\left| \nabla u\right|\leq\frac{C\left( \left|\ln u\right|+1\right)u^{1-\frac{1}{2q}}}{r}\leq\frac{C}{r}.$$
		Then standard bootstrapping arguments give the desired regularity (cf. \cite[section 8]{GT}).
	\end{proof}
	
	From the view of specific geometric problems, integrable solutions to (\ref{1.1}) with $a(x)>0$ are the most important ones such as the minimizers of $\mathcal{W}$ entropy. Based on the previous estimates, we find some further properties of these solutions on non-compact manifolds.
	\begin{corollary}\label{c3.3}
		Let $u>0$ be a positive solution to the equation (\ref{1.1}) on a complete non-compact and non-collapsing Riemannian manifold $(M^n,g)$ with $Ric\geq-(n-1)Kg$ for some $K\geq0$. Suppose on $(M^n,g)$, $0<A_1\leq a\leq A_2$, $\left|\nabla a\right|\leq A_3$, $(\Delta a)^-\leq A_4$, $\left|b\right|\leq B_1$, $\left|\nabla b\right|\leq B_2$ and $\int_{M^n}u^k<+\infty$ for some $k>0$, then $u(x)\longrightarrow0$ as $x\longrightarrow\infty$ uniformly. Especially, in this situation, $u$ has a maximum point $x_0$ such that
		$$u(x_0)\geq e^{-\frac{B_1}{A_1}}.$$
	\end{corollary}
	\begin{proof}
		Since $\int_{M^n}u^k<+\infty$, then for $x\longrightarrow\infty$,
		\begin{equation}\label{3.13}
			\int_{B(x,1)}u^k\longrightarrow0.
		\end{equation}
		However, by Theorem \ref{t1.1} and \ref{t1.2}, $\left|\nabla u\right|$ is uniformly bounded on $(M^n,g)$. Hence if there is a sequence $x_k\longrightarrow\infty$ such that $u\geq c>0$, then by the gradient estimate and the non-collapsing condition, $\int_{B(x,1)}u^k$ has a strict positive lower bound which contradicts (\ref{3.13}). Hence $u\longrightarrow0$ at infinity.
		
		Moreover, since $u\longrightarrow0$ at infinity, $u$ must have a maximum point $x_0$ and then by maximum principle, $\Delta u(x_0)\leq 0$, hence
		$$a(x_0)\ln u(x_0)+b(x_0)\geq0.$$
		Finally,
		$$u(x_0)\geq e^{-\frac{b(x_0)}{a(x_0)}}\geq e^{-\frac{B_1}{A_1}}.$$
	\end{proof}
	
	In form, the solution $u$ to equation (\ref{1.1}) is a generalization of positive harmonic functions, so one may expect that its gradient behaves like that of positive harmonic functions. Especially, one may ask whether there exists a constant $C>0$ such that on $(M^n,g)$ with $Ric\geq-(n-1)Kg$,
	\begin{equation}\label{3.14}
		\frac{\left|\nabla u\right|}{u}\leq C.
	\end{equation}
	As a supplement of our Theorem \ref{t1.2}, in the following, we claim that (\ref{3.14}) is impossible if $u\longrightarrow0$.
	\begin{proposition}\label{p3.3}
		Let $u>0$ be a positive solution to the equation (\ref{1.1}) on $B(x,r)\subset(M^n,g)$ with $0<r\leq 1$ and $Ric\geq-(n-1)Kg$ for some $K\geq0$. Suppose on $B(x,r)$, $0<A_1\leq a\leq A_2$, $\left|b\right|\leq B_1$ and
		\begin{equation}\label{3.15}
			\frac{\left|\nabla u\right|}{u}\leq C_1,
		\end{equation}
		then there exists a constant $C=C\left( n,K,A_1,B_1,C_1\right)>0$ such that
		\begin{equation}\label{3.15*}
			u(x)\geq e^{-C}.
		\end{equation}
		Consequently, there is no $u$ such that (\ref{3.15}) is valid but $u\longrightarrow0$. Especially, if (\ref{3.15}) holds on $(M^n,g)$, there is no $u$ satisfying the conditions of Corollary \ref{c3.3}.
	\end{proposition}
	\begin{proof}
		Let $w=\ln u$. Then by Green's formula,
		\begin{align}\label{3.16}
			\int_{B(x,r)}\Delta w=\int_{\partial B(x,r)}\left\langle \textbf{n}, \nabla w\right\rangle,
		\end{align}
		where $\textbf{n}$ is the outward normal vector of $\partial B(x,r)$. By volume comparison theorem and co-area formula, as well as mean value theorem for integral, for some $C_2=C_2(n,K)>0$, we have $\left|\partial B(x,r)\right|\leq C_2\left| B(x,r)\right|$. Since $w$ satisfies equation (\ref{2.2}) and $\left|\nabla w\right|\leq C_1$, therefore by Jensen's inequality,
		\begin{equation}\label{3.17*}
			\ln\left(\fint_{B(x,r)}u\right)\geq\fint_{B(x,r)}w\geq-C.
		\end{equation}
		Utilizing $\left|\nabla w\right|\leq C_1$ again, it's easy to know the following Harnack inequality for $y\in B(x,r)$:
		\begin{equation}\label{3.18*}
			\frac{u(y)}{u(x)}\leq e^{C_1}.
		\end{equation}
		Then combining (\ref{3.17*}) and (\ref{3.18*}) gives (\ref{3.15*}).
	\end{proof}
	
	\section{A Liouville type theorem}\label{sec4}
	In this section, we focus on equation (\ref{1.1}) with constant coefficients $a(x)>0$ and $ b(x)$. In this situation, (\ref{1.1}) becomes
	\begin{equation}\label{4.1}
		\Delta u(x)+au(x)\ln u(x)+bu(x)=0,
	\end{equation}
	and has a constant solution $u(x)\equiv e^{-\frac{b}{a}}$. First we observe that if the upper bound of $u$ is small enough, then $u\equiv0$.
	\begin{proposition}
		Let $u>0$ be a positive solution to the equation (\ref{4.1}) on non-compact $(M^n,g)$ with $Ric\geq0$. Then there is no $u$ such that $u\leq e^{-\frac{b}{a}-1}$.
	\end{proposition}
	\begin{proof}
		If $u\leq e^{-\frac{b}{a}-1}$, then by (\ref{4.1}), $\Delta u\geq0$. On the other hand, we compute
		\begin{align*}
			\left|\nabla u\right|\Delta\left|\nabla u\right|+\left\langle \nabla\left|\nabla u\right|, \nabla\left|\nabla u\right| \right\rangle=\frac{1}{2}\Delta\left|\nabla u\right|^2&=\left|D^2u\right|^2+\left\langle \nabla u, \nabla\Delta u\right\rangle+Ric\left(\nabla u,\nabla u\right)\\
			&\geq \left|D^2u\right|^2-\left\langle \nabla u, \nabla\left(au\ln u+bu\right)\right\rangle\\
			&\geq \left|D^2u\right|^2-\left(a+a\ln u+b\right)\left|\nabla u\right|^2\\
			&\geq \left|D^2u\right|^2,
		\end{align*}
		hence $\Delta\left|\nabla u\right|\geq0$. Then by the mean value inequality for subharmonic functions due to  Li-Schoen \cite{LS}(see also \cite[chapter 2,Theorem 6.1]{SY}), there exists a uniform constant $C_1>0$ such that for any $B(x,r)$,
		$$\left|\nabla u\right|^2(x)\leq C_1\fint_{B(x,r)}\left|\nabla u\right|^2.$$
		On the other hand, by the integral estimate of non-negative sub-harmonic functions \cite[chapter 2, Lemma 6.3]{SY}, there exists a uniform constant $C_2>0$ such that
		$$\left|\nabla u\right|^2(x)\leq\frac{C_2}{r^2}\fint_{B(x,r)}u^2.$$
		In light of Theorem \ref{t1.1}, we know $u\leq D$ where $D$ is a uniform upper bound of $u$, hence by the above arguments, there exists a uniform constant $C>0$ such that
		$$\left|\nabla u\right|^2(x)\leq\frac{C}{r^2}.$$
		Then setting $r\longrightarrow +\infty$ implies $u$ must be constant and contradicts $0<u\leq e^{-\frac{b}{a}-1}<e^{-\frac{b}{a}}$.
	\end{proof}
	
	The above result implies that around $u\equiv0$, there is no positive solution to equation (\ref{4.1}). In fact, this type of gap phenomenon also occurs to the constant solution $u\equiv e^{-\frac{b}{a}}$.
	\begin{theorem}[=Theorem \ref{t1.3}]\label{t4.1}
	For constant coefficients $a>0$ and $b$, there exists a constant $\varepsilon=\varepsilon\left(n,a,b\right)>0$ such that if
	\begin{equation}\label{4.2*}
		0<e^{-\frac{b}{a}}-\varepsilon\leq u\leq e^{-\frac{b}{a}}\quad\mbox{or}\quad e^{-\frac{b}{a}}\leq u\leq e^{-\frac{b}{a}}+\varepsilon
	\end{equation}
	is a solution to (\ref{1.1}) on $B\left(x,\varepsilon^{-1}\right)$ with $Ric\geq 0$, then $u\equiv e^{-\frac{b}{a}}$. Especially, in these situations, except for $u\equiv e^{-\frac{b}{a}}$, there is no nonconstant solution $u$, which satisfies the above pinching condition (\ref{4.2*}), such that $u\rightarrow e^{-\frac{b}{a}}$ at infinity on a complete non-compact Riemannian manifold with $Ric\geq 0$.
\end{theorem}
	\begin{proof}
		Since the proofs of two cases are almost the same, here we only need to give the proof of the later case, i.e. $e^{-\frac{b}{a}}\leq u\leq e^{-\frac{b}{a}}+\varepsilon$.
		For $u> e^{-\frac{b}{a}}$, we rewrite equation (\ref{4.1}) as
		\begin{equation}\label{4.2}
			\Delta\left(u-e^{-\frac{b}{a}}\right)+V\left(u-e^{-\frac{b}{a}}\right)=0,
		\end{equation}
		where
		$$V=\frac{u(a\ ln u+b)}{\left(u-e^{-\frac{b}{a}}\right)}.$$
		Moreover, it's easy to check that
		$\lim\limits_{u\longrightarrow e^{-\frac{b}{a}}}V(u)=a>0$ and $\lim\limits_{u\longrightarrow e^{-\frac{b}{a}}}V'(u)=\frac{a}{2e^{-\frac{b}{a}}}>0$, hence if we set $w=u-e^{-\frac{b}{a}}\geq 0$, then by (\ref{4.1}) and (\ref{4.2}), the following Schr\"{o}dinger equation is well defined for all $w\geq0$:
		\begin{equation}\label{4.3}
			\Delta w+Vw=0.
		\end{equation}
		Also, we may choose a small $\varepsilon=\varepsilon(a,b)>0$ such that $V\geq \frac{2a}{3}$.
		As a consequence of the strong maximum principle, either $w\equiv0$ or $w>0$ in $B\left(x,\varepsilon^{-1}\right)$, where $\varepsilon$ is determined later, therefore we only need to consider the case of $w>0$.

		Similarly, if $u\leq e^{-\frac{b}{a}}$, then we set $w=e^{-\frac{b}{a}}-u\geq 0$ and consider
		\begin{equation*}
			\Delta w+Uw=0,
		\end{equation*}
		where
		$$U=\frac{-u(a\ ln u+b)}{\left(e^{-\frac{b}{a}}-u\right)}.$$
		
		As before, let $v=\ln w=\ln\left(u-e^{-\frac{b}{a}}\right)$ and $G=\left| \nabla v\right|^2+V$, then by (\ref{4.3}),
		\begin{align}\label{4.4}
			\Delta v+G=0.
		\end{align}
		First we compute
		\begin{align*}
			\Delta V&=\Delta\left(\frac{u(a\ ln u+b)}{\left(u-e^{-\frac{b}{a}}\right)}\right)\\
			&=\frac{\Delta \left( u(a\ ln u+b)\right)}{u-e^{-\frac{b}{a}}}+ u(a\ ln u+b)\Delta \left(u-e^{-\frac{b}{a}}\right)^{-1}+2\left\langle\nabla \left( u(a\ ln u+b)\right), \nabla \left(u-e^{-\frac{b}{a}}\right)^{-1} \right\rangle\\
			&=\left\lbrace\frac{a(1+\ln u)+b}{u-e^{-\frac{b}{a}}}-\frac{u(a\ln u+b)}{\left(u-e^{-\frac{b}{a}} \right)^2}\right\rbrace\Delta u\\
			&\quad+\left\lbrace\frac{a}{u\left( u-e^{-\frac{b}{a}}\right)}+\frac{2u(a\ln u+b)}{\left(u-e^{-\frac{b}{a}}\right)^3}-\frac{2a(1+\ln u)+2b}{\left(u-e^{-\frac{b}{a}} \right)^2}\right\rbrace\left| \nabla u\right|^2\\
			&=: \uppercase\expandafter{\romannumeral1}+\uppercase\expandafter{\romannumeral2}.
		\end{align*}
		Recall $\Delta u=-au\ln u-bu$, then we can rewrite  $\uppercase\expandafter{\romannumeral1}=V^2-(a+a\ln u+b)V$. Notice that
		$$\lim\limits_{u\longrightarrow e^{-\frac{b}{a}}}\frac{\uppercase\expandafter{\romannumeral1}}{V}=\lim\limits_{u\longrightarrow e^{-\frac{b}{a}}}\left(V-(a+a\ln u+b)\right)=0,$$
		hence it's easy to know, for all $\varepsilon=\varepsilon(a,b)>0$ small enough, if $u\leq e^{-\frac{b}{a}}+\varepsilon$, then
		\begin{equation}\label{4.5}
			\frac{\uppercase\expandafter{\romannumeral1}}{V}\geq -C(\varepsilon,a,b),
		\end{equation}
		and moreover, $C\longrightarrow0$ as $\varepsilon\longrightarrow0$.
		
		As for $\uppercase\expandafter{\romannumeral2}$, we consider $$\overline{\uppercase\expandafter{\romannumeral2}}=\frac{a}{u}+\frac{2u(a\ln u+b)}{\left(u-e^{-\frac{b}{a}}\right)^2}-\frac{2a(1+\ln u)+2b}{\left(u-e^{-\frac{b}{a}} \right)}.$$
		Also, it's easy to check
		$\lim\limits_{u\longrightarrow e^{-\frac{b}{a}}}\overline{\uppercase\expandafter{\romannumeral2}}=0$. Hence for $\left| \nabla u\right|>0$ and all $\varepsilon=\varepsilon(a,b)>0$ small enough, if $u\leq e^{-\frac{b}{a}}+\varepsilon$, then
		\begin{align}\label{4.6}
			\frac{\uppercase\expandafter{\romannumeral2}}{\left| \nabla v\right|^2}=\overline{\uppercase\expandafter{\romannumeral2}}\left(u-e^{-\frac{b}{a}} \right)\geq-C(\varepsilon,a,b),
		\end{align}
		and also, $C\longrightarrow0$ as $\varepsilon\longrightarrow0$.
		
		Now by (\ref{4.5}), (\ref{4.6}) and recall $G=\left| \nabla v\right|^2+V$, we have
		\begin{align}\label{4.7}
			\Delta V\geq-C(\varepsilon,a,b)G.
		\end{align}
		Next, since $Ric\geq0$, then by (\ref{4.4}), (\ref{4.7}), we estimate
		\begin{align}\label{4.8}
			\Delta G&=\Delta\left(\left| \nabla v\right|^2+V\right)\nonumber\\
			&=\Delta \left| \nabla v\right|^2+\Delta V\nonumber\\
			&\geq 2\left|D^2 v\right|^2+2\left\langle\nabla v,\nabla\Delta v\right\rangle-C(\varepsilon,a,b)G\nonumber\\
			&\geq \frac{2G^2}{n}-2\left\langle\nabla v,\nabla G\right\rangle-C(\varepsilon,a,b)G\nonumber\\
			&\geq \frac{G^2}{n}-2\left\langle\nabla v,\nabla G\right\rangle,
		\end{align}
		here we choose $\varepsilon=\varepsilon(a,b,n)$ sufficiently small such that $G\geq \frac{2a}{3}
		\geq C(\varepsilon,a,b)n$.
		Next, for $r>0$, $l\geq0$ and $\eta\in C^\infty_0(B_{r})$, we multiply by $\eta^2G^l$ on both sides of (\ref{4.8}) and integrate on $B_{r}$. Since $V>0$, then by the arguments as to obtain (\ref{2.12}),
		\begin{align*}
			&\int\left| \nabla\left( \eta G^{\frac{l+1}{2}}\right) \right|^2\nonumber\\
			&\leq\frac{\left( l+1\right)^2 }{2(l-1)}\int \left(\frac{4}{l+1}-\frac{1}{n}\right)\eta^2G^{l+2}+\frac{l\left(2l+3\right)+(l+1)^2}{l(l-1)}\int \left| \nabla\eta\right|^2G^{l+1}.
		\end{align*}
		Therefore by the Sobolev inequality (\ref{2.1}) with $\varphi=\eta G^{\frac{l+1}{2}}$,
		\begin{align}\label{4.9}
			&\left( \fint_{B_r}\left(\eta^2 G^{l+1}\right)^{\frac{n}{n-2}}\right)^{\frac{n-2}{n}}\nonumber\\
			&\leq C_Sr^2l\fint_{B_r}\left(\frac{4}{l+1}-\frac{1}{n}\right)\eta^2G^{l+2}+C_S\fint_{B_r}\eta^2 G^{l+1}+C_Sr^2l\fint_{B_r}\left| \nabla\eta\right|^2G^{l+1}.
		\end{align}
		Let $r\geq 1$ and $l\geq 8n+\frac{2n}{a}$, then (\ref{4.9}) gives
		\begin{align}\label{4.10}
			&\left( \fint_{B_r}\left(\eta^2 G^{l+1}\right)^{\frac{n}{n-2}}\right)^{\frac{n-2}{n}}\nonumber\\
			&\leq C_Sr^2l\fint_{B_r}\left| \nabla\eta\right|^2G^{l+1}.
		\end{align}
		Consequently, by standard iteration as to get (\ref{2.17}), there exists a constant $C_1=C_1(n,a)>0$ independent of $r$ such that
		\begin{equation}\label{4.11}
			\sup\limits_{B_\frac{r}{2}}G\leq C_1\left\|G\right\|^*_{1, B_\frac{4r}{5}}.
		\end{equation}
		Next, we choose some $\eta\in C^\infty_0(B_r)$ such that
		$$\eta\equiv1\quad\mbox{on}\hspace*{0.3em}B_\frac{4r}{5}\quad\mbox{and}\quad \left| \nabla\eta\right|\leq\frac{10}{r}.$$
		For $\eta^2G$, since $V>0$, there holds
		\begin{align*}
			\int\eta^2G=-\int\eta^2\Delta v&=2\int\left\langle \nabla\eta, \eta\nabla v\right\rangle\nonumber\\
			&\leq2\int \left| \nabla\eta\right|^2+\frac{1}{2}\int\eta^2\left| \nabla v\right|^2\nonumber\\
			&=2\int\left|\nabla\eta\right|^2+\frac{1}{2}\int\eta^2G-\frac{1}{2}\int\eta^2V\\
			&\leq 2\int \left| \nabla\eta\right|^2+\frac{1}{2}\int\eta^2G.
		\end{align*}
		Therefore,
		\begin{align}\label{4.12}
			\int\eta^2G\leq4\int \left| \nabla\eta\right|^2.
		\end{align}
		As a consequence of (\ref{4.12}),
		\begin{align}\label{4.13}
			0<\left\|G\right\|^*_{1, B_\frac{4r}{5}}\leq\frac{\int\eta^2G}{\left| B_\frac{4r}{5}\right|}\leq\frac{400\left| B_r\right|}{\left| B_\frac{4r}{5}\right|r^2}.
		\end{align}
		Finally, combining the volume comparison theorem and (\ref{4.13}), if we set $r= r(n,a)$ large enough, then $G\leq \frac{a}{2}$, but this contradicts the fact that $G\geq\frac{2a}{3}>0$ at the beginning. Hence we complete the proof.
	\end{proof}
	
	\noindent {\it\textbf{Acknowledgements}}: The author Y. Wang is supported partially by NSFC (Grant No.11971400) and National key Research and Development projects of China (Grant No. 2020YFA0712500).
	
	\bibliographystyle{amsalpha}

	Jie Wang,  Institute of Geometry and Physics, University of Science and Technology of China, No. 96 Jinzhai Road, Hefei, Anhui Province, 230026, China.
	
	Email: wangjie9math@163.com
	
	Youde Wang, 1. School of Mathematics and Information Sciences, Guangzhou University; 2. Hua Loo-Keng Key Laboratory of Mathematics, Institute of Mathematics, Academy of Mathematics and Systems Science, Chinese Academy of Sciences, Beijing 100190, China; 3. School of Mathematical Sciences, University of Chinese Academy of Sciences, Beijing 100049, China.
	
	Email: wyd@math.ac.cn
	
\end{document}